\documentclass[11pt,a4paper]{article}
\usepackage[utf8]{inputenc}

\usepackage{amssymb,amscd,amsfonts,amsbsy,amsmath,amsthm,amsrefs}
\usepackage{dsfont}
\usepackage{enumerate}
\usepackage{mathrsfs}
\usepackage{epsf,epsfig,esint}
\usepackage{pdfsync}
\usepackage[all]{xy}
\usepackage{pdfpages}

\newtheorem{proposition}{Proposition}[section]
\newtheorem{theorem}[proposition]{Theorem}
\newtheorem{lemma}[proposition]{Lemma}

\theoremstyle{definition}

\theoremstyle{remark}
\newtheorem{remark}[proposition]{Remark}

\numberwithin{equation}{section}

\newcommand{\R}{\mathds{R}}
\renewcommand{\P}{\mathbb{P}}
\newcommand{\C}{\mathscr{C}}
\DeclareMathOperator{\esssup}{ess\,sup}

\begin{document}

\title{Well-posedness for the Boussinesq system\\in critical spaces via maximal regularity}
\author{Lorenzo Brandolese \and Sylvie Monniaux}
\date{}


\maketitle

\abstract{We establish the existence and the uniqueness for the Boussinesq system
in $\R^3$ in the critical space $\C([0,T],L^3(\R^3)^3)\times L^2(0,T;L^{3/2}(\R^3))$.}

\section{Introduction}

We consider the Cauchy problem associated with the Boussinesq system in ${\mathds{R}}^3$:
\begin{equation}\tag{B}
\label{B}
\begin{array}{rclcc}
\partial_t u-\Delta u+\nabla\pi+\nabla\cdot(u\otimes u)&=&\theta e_3 &\mbox{ in }&(0,T)\times{\mathds{R}}^3 \\
{\rm div}\,u&=&0 &\mbox{ in }&(0,T)\times{\mathds{R}}^3 \\
\partial_t \theta -\Delta \theta +u\cdot \nabla\theta &=&0 &\mbox{ in }&(0,T)\times{\mathds{R}}^3, 
\end{array}
\end{equation}
where $u$ denotes the velocity of the fluid, $\pi$ the pressure, $\theta$ the temperature and $e_3$ the third vector 
of the canonical basis in $\R^3$.
The given initial velocity and temperature are denoted respectively $u_0$ and 
$\theta_0$. The  initial velocity will be always assumed to satisfy the condition ${\rm div}\,u_0=0$.
The system~\eqref{B} appears in the study of the motion of incompressible viscous flows when
one takes into account buoyancy effects arising from temperature variations inside the fluid. When the latter are 
neglected ($\theta\equiv0$), the system boils down to the classical Navier--Stokes equations.  

The natural scaling leaving the Boussinesq system invariant is
$\lambda\mapsto(u_\lambda,\theta_\lambda)$ with 
\[
u_\lambda(t,x)= \lambda u(\lambda^2t,\lambda x)
\qquad
\text{and}
\qquad
\theta_\lambda(t,x)= \lambda^3 \theta (\lambda^2t,\lambda x).
\]
This motivates the study of~\eqref{B} in function spaces that are left invariant
by the above scaling.
Assuming $(u_0,\theta_0)\in L^3(\R^3)^3\times L^1(\R^3)$, an adaptation of Kato's $L^p$-theory on strong solutions of 
Navier--Stokes~\cite{Kat84}, yields the local-in-time existence and the uniqueness in appropriate scale-invariant function 
spaces where the fixed point argument applies.
But, as discussed in \cite{BraH}, the uniqueness problem in the natural space
$\C([0,T],L^3(\R^3)^3)\times \C([0,T],L^1(\R^3))$
seems to be out of reach, due to the lack of regularity
results in this class, and to the difficulty of giving a meaning, in the distributional sense, to the nonlinearity  
$\nabla\cdot(u\theta)$ (and, of course, to $u\cdot \nabla \theta$) in the last equation in~\eqref{B}.
To circumvent this difficulty, the uniqueness for the Cauchy problem, in~\cite{BraH}, was established
in a smaller space for the temperature, namely
$\C([0,T],L^1(\R^3)\cap L^\infty_{\rm loc}(0,T;L^{q,\infty}(\R^3))$, for some $q>3/2$.
This restriction on the temperature, however, is a bit artificial:
the excluded borderline case, $q=3/2$, is precisely the most interesting one, as it corresponds
to the minimal regularity to be imposed on the temperature, when the velocity
is in the natural space $L^3(\R^3)^3$, to give a sense to the nonlinearity.

The scaling relations then lead us to consider solutions such that
$t\mapsto \|\theta(t)\|_{L^{3/2}}$ is in $L^2$.
Therefore, it seems natural to address the uniqueness problem (and the existence) in
\[
\C([0,T],L^3(\R^3)^3)\times L^2(0,T;L^{\frac32}(\R^3)).
\]
As the Lorentz-space approach of \cites{KozY, Mey, DanP}, applied in \cite{BraH}, 
fails when $q=3/2$, we have to adopt a different strategy.
Our main tools will be maximal regularity estimates. The idea of using the maximal regularity in uniqueness 
problems goes back to~\cite{Mon}, where the second author
gave a short proof of celebrated Furioli, Lemari\'e and Terraneo's uniqueness theorem \cite{FLT} of mild solutions of the 
Navier--Stokes equations in $\C([0,T],L^3(\R^3))$.
In the present paper,  we will need to use the maximal regularity in an original way,
in order to make it applicable despite the product $u\theta$ a priori just belongs to $L^1(\R^3)^3$.

In fact, our approach allows us to obtain the uniqueness, and then the regularity as a byproduct of the existence theory, 
in a larger class, namely
\begin{equation}
\label{eq:L-infty-class}
\Bigl(\C([0,T],L^3(\R^3)^3) +r\, L^\infty(0,T;L^3(\R^3)^3)\Bigr)\times L^2(0,T;L^{\frac32}(\R^3))
\end{equation}
for some small enough $r>0$.

One could speculate that the smallness condition on the parameter $r$ may be unessential and that the uniqueness and the 
regularity could be true in the larger space $L^\infty(0,T;L^3(\R^3)^3)\times L^2(0,T;L^{\frac32}(\R^3))$. This would be a nontrivial 
generalization for the system~\eqref{B} of the deep result of Escauriaza, Seregin and \v Sverák \cite{ISS}, about endpoint 
Serrin regularity criteria for the Navier--Stokes equations. Establishing such a result would probably require various ingredients 
(backward uniqueness, profile decompositions, \cites{ISS,GKP,Phuc}, etc.). 
Whether or not such a stronger statement is true, we feel that our main theorem
would remain of interest because of its attractive proof, entirely based on maximal
regularity estimates.

\section{Statement of the main results}

Let $T>0$ and $r>0$. Let ${\mathscr{S}}'(\R^3)$ denote the dual of Schwartz space.
In order to state our uniqueness result in the class
\[
X_{T,r}:=\Bigl(\C([0,T],L^3(\R^3)^3) +r\, L^\infty(0,T;L^3(\R^3)^3)\Bigr)
\times L^2(0,T;L^{\frac{3}{2}}({\mathds{R}}^3)),
\]
we first clarify what we mean by solution of~\eqref{B}.
By definition, a mild solution of \eqref{B} with inital data $(u_0,\theta_0)\in \mathscr{S}'(\R^3)^3\times\mathscr{S}'(\R^3)$, 
and ${\rm div\,}u_0=0$, is a couple $(u,\theta)\in X_{T,r}$ solving the Boussinesq system written in 
its integral form~\eqref{def:sol} below:
\begin{equation}
\label{def:sol}
\begin{split}
u&= a+B(u,u)+L(\theta)\\
\theta&= b+C(u,\theta)
\end{split}
\end{equation}
with 
\[
a(t)=e^{t\Delta}u_0,  \qquad
 b(t)=e^{t\Delta}\theta_0.
\]
The operators $B$, $C$ and $L$ are defined, for $t\in[0,T]$, and $(u,\theta)\in X_{T,r}$,
by 
\begin{equation}
\label{def:B}
B(u,v)(t)=-\int_0^te^{(t-s)\Delta}{\mathbb{P}}\bigl(\nabla\cdot\bigl(u(s)\otimes v(s)\bigr)\bigr)\,{\rm d}s, 
\end{equation}
\begin{equation}
\label{def:C}
C(u,\theta)(t)=-\int_0^t e^{(t-s)\Delta}{\rm div}\,\bigl(\theta(s)u(s)\bigr)\,{\rm d}s,
\end{equation}
and
\begin{equation}
\label{def:L}
L(\theta)(t)=\int_0^t e^{(t-s)\Delta}{\mathbb{P}}(\theta(s)e_3)\,{\rm d}s.
\end{equation}
Here $\P$ denotes Leray's projector onto divergence-free vector fields and $\bigl(e^{t\Delta}\bigr)_{t\ge 0}$
is the heat semigroup.

Our main result then is stated as follows:

\begin{theorem}
\label{th:uni}
There is an absolute constant $r_0>0$ such that, if 
$(u_1,\theta_1)$ and $(u_2,\theta_2)$ are two mild solutions to~\eqref{B} 
in~$X_{T,r}$, with the same initial data $(u_0,\theta_0)\in \mathscr{S}'(\R^3)^3\times \mathscr{S}'(\R^3)$, 
${\rm div}\,u_0=0$ and $0\le r< r_0$,
then
\[
(u_1,\theta_1)=(u_2,\theta_2).
\]
\end{theorem}

As we will see, any solution $(u,\theta)$ as in~Theorem~\ref{th:uni} must belong
to $\C([0,T],\mathscr{S}'(\R^3)^3\times \mathscr{S}'(\R^3))$, and the initial data must belong more precisely to 
$L^3(\R^3)^3\times B^{-1}_{3/2,2}(\R^3)$.
The above theorem is then completed by the corresponding existence result:

\begin{theorem}
\label{th:existence}\mbox{}
\begin{itemize}
\item[i)]
Let $(u_0,\theta_0)\in L^3(\R^3)^3\times B^{-1}_{3/2,2}(\R^3)$, with ${\rm div}\,u_0=0$.
Then there exists $T>0$ and a solution of~\eqref{def:sol}
$(u,\theta)\in\C([0,T],L^3(\R^3)^3)\times L^2(0,T;L^{\frac32}(\R^3))$.
\item[ii)]
If $\theta_0$ belongs to the smaller homogeneous Besov space 
$\dot B^{-1}_{3/2,2}(\R^3)$ and if $\|u_0\|_{L^3}+\|\theta_0\|_{\dot B^{-1}_{3/2,2}}$ is small enough, than such solution is global and
$(u,\theta)\in\C_b(0,\infty;L^3(\R^3)^3)\times L^2(0,\infty;L^{\frac32}(\R^3))$.
\end{itemize}
\end{theorem}

For some other existence results for the Boussinesq system in different functional setting we refer, e.g.,
to~\cite{DanPbis, DanP, KarP}.

\section{Applications of the maximal regularity}
\label{sec:max-reg}

The purpose of this section is to study the properties of the operators $B$, $C$ and $L$, respectively defined by
\eqref{def:B}, \eqref{def:C} and \eqref{def:L}, by means of the following maximal regularity result.
The theorem below is classical. See \cite{deS} for the case $p \in (1,\infty)$ and $q=2$. The general case was first
proved in \cite{LSU68}, Chapter IV, \S3. The proof of the estimates for the mixed fractional space-time derivatives 
goes back to~\cite[Theorem 6]{Sob}. See also \cite[Theorem~7.3]{Lem02} for a modern proof.

\begin{theorem}[Maximal regularity]
\label{thm:regmax}
Let $1<p,q<\infty$.
Let $R$ be the operator defined for $f\in L^1_{\rm loc}(0,\infty;{\mathscr{S}}'(\R^d))$,
$d\ge1$, by
\begin{equation}
\label{def:R}
Rf(t)=\int_0^te^{(t-s)\Delta}f(s)\,{\rm d}s,\quad t>0.
\end{equation}
Such operator $R$ is bounded from $L^p\bigl(0,\infty;L^q({\mathds{R}}^d)\bigr)$ to 
$\dot W^{1,p}\bigl(0,\infty;L^q({\mathds{R}}^d)\bigr)
\cap L^p\bigl(0,\infty;\dot W^{2,q}({\mathds{R}}^d)\bigr)$. 
In other words,
$\frac{{\rm d}}{{\rm d}t} R$, $\Delta R$, and 
$(-\Delta)^{\alpha}\bigl(\tfrac{{\rm d}}{{\rm d}t}\bigr)^{1-\alpha}R$,
for any $0<\alpha<1$, are bounded operators in   
$L^p\bigl(0,\infty;L^q({\mathds{R}}^d)\bigr)$.
Moreover, there exists a constant $C_{p,q}$ such that 
\[
\bigl\|\tfrac{{\rm d}}{{\rm d}t} Rf\bigr\|_{L^p(L^q)}+\bigl\|\Delta Rf\bigr\|_{L^p(L^q)}
+\bigl\|(-\Delta)^{\alpha}\bigl(\tfrac{{\rm d}}{{\rm d}t}\bigr)^{1-\alpha}Rf\bigr\|_{L^p(L^q)} \le C_{p,q}\|f\|_{L^p(L^q)},
\]
for all $\alpha \in (0,1)$.
\end{theorem}

To establish Theorem~\ref{th:uni}, we assume that we have two
mild solutions $(u_1,\theta_1)\in X_{T,r}$ and $(u_2,\theta_2)\in X_{T,r}$
of~\eqref{B}, arising from the same initial datum $(u_0,\theta_0)\in \mathscr{S'}(\R^3)^3\times \mathscr{S}'(\R^3)$.
Letting $u=u_1-u_2$ and $\theta=\theta_1-\theta_2$, we then obtain
$(u,\theta)\in X_{T,r}$ and
\begin{equation}
\label{eq:unicite}
\begin{split}
u&=B(u,u_1)+B(u_2,u)+L(\theta),\\
\theta&=C(u_1,\theta)+C(u,\theta_2).
\end{split}
\end{equation}

The maximal regularity theorem allows us to obtain all the relevant estimates for the
operators $B$ and $C$ and~$L$.

\begin{proposition}
\label{prop:B}
For all $\varepsilon>0$, there exists $r>0$ such that for all 
$v,w\in \C([0,T],L^3(\R^3)^3) +r\, L^\infty([0,T],L^3(\R^3)^3)$, 
and for all $1<p<\infty$,
there exists 
$\tau=\tau(\varepsilon,p,v,w)>0$ such that the linear operator
\begin{align}
B(\cdot,v)+B(w,\cdot)\colon &L^4(0,\tau; L^6({\mathds{R}}^3)^3)\longrightarrow L^4(0,\tau; L^6({\mathds{R}}^3)^3),
\label{eq:L4L6}\\
B(\cdot,v)+B(w,\cdot)\colon &L^p(0,\tau; L^3({\mathds{R}}^3)^3)\longrightarrow L^p(0,\tau; L^3({\mathds{R}}^3)^3)
\label{eq:LpL3}
\end{align}
is bounded, with operator norm less than~$\varepsilon$.
\end{proposition}

\begin{proof}
Let $r>0$, to be chosen later. 
For $v,w\in \C([0,T],L^3(\R^3)^3) +r\, L^\infty(0,T;L^3(\R^3)^3)$, we can find
$v_r, w_r \in \C_c([0,T]\times\mathds{R}^3)$ such that 
\begin{equation}
\label{eq:vr,wr}
\esssup_{[0,T]}\|v-v_r\|_{L^3}+\esssup_{[0,T]}\|w-w_r\|_{L^3}\le 3r.
\end{equation}
Let us introduce the functions~$f$ and $g$ defined by
\begin{equation}
\label{eq:def-f}
f(s)=(-\Delta)^{-1}{\mathbb{P}}\bigl(\nabla\cdot\bigl(u(s)\otimes (v-v_r)(s) + (w-w_r)(s)\otimes u(s)\bigr)\bigr), \quad s\in[0,T]
\end{equation}
and 
\[
g(s)=-(-\Delta)^{-3/4}{\mathbb{P}}\bigl(\nabla\cdot
\bigl(u(s)\otimes v_r(s) +w_r(s)\otimes u(s)\bigr)\bigr), \quad s\in[0,T].
\]
We have
\begin{equation}
\label{eq:B-R}
B(u,v)+B(w,u)=\Delta Rf +(-\Delta)^{3/4}Rg,
\end{equation}
where $R$ is the vector-valued analogue of the scalar operator defined in Theorem~\ref{thm:regmax}.

\begin{itemize}
\item[(i)]
Let us first consider~\eqref{eq:L4L6}.
We easily see that the norm of $f$ in $L^4(0,\tau; L^6({\mathds{R}}^3)^3)$ 
is bounded by the norm 
of $u\otimes (v-v_r)+(w-w_r)\otimes u$ in 
$L^4(0,\tau; L^{2}({\mathds{R}}^3)^{3\times 3})$.
Indeed, the operator 
$(-\Delta)^{-1}{\mathbb{P}}\nabla\cdot$ is bounded from 
$L^2({\mathds{R}}^3)^{3\times 3}$ to $L^6({\mathds{R}}^3)^3$.
Hence, 
\begin{align*}
\|\Delta Rf\|_{L^4(0,\tau;L^6({\mathds{R}}^3)^3)}
&\le C_{4,6} \|u\otimes (v-v_r)+(w-w_r)\otimes u\|_{L^4(0,\tau; L^2({\mathds{R}}^3)^{3\times 3})}
\\
&\le 3rC_{4,6}  \|u\|_{L^4(0,\tau; L^6({\mathds{R}}^3)^3)}.
\end{align*}
The norm of $g$ in $L^4(0,\tau; L^6({\mathds{R}}^3)^3)$ is bounded by the norm of 
$u\otimes v_r+w_r\otimes u$ in 
$L^4(0,\tau; L^3({\mathds{R}}^3)^{3\times 3})$.
To see this, first observe that the operator $(-\Delta)^{-3/4}{\mathbb{P}}\nabla\cdot$ is bounded from 
$L^3({\mathds{R}}^3)^{3\times 3}$ to $L^6({\mathds{R}}^3)^3$.
Moreover, 
$\|(-\Delta)^{3/4}e^{t\Delta}\|_{\mathscr{L}(L^6(\R^3)^3)}\lesssim t^{-3/4}$. 
As $(-\Delta)^{3/4}R$ is a convolution operator, we have
\begin{align*}
\|(-\Delta)^{3/4}Rg\|_{L^4(0,\tau;L^6({\mathds{R}}^3)^3)}&\le 
c\,\|t\mapsto (-\Delta)^{3/4}e^{t\Delta}\|_{L^1(0,\tau;{\mathscr{L}}(L^6({\mathds{R}}^3)^3))}\|g\|_{L^4(0,\tau; L^6({\mathds{R}}^3)^3)}
\\
&\le c'\, \tau^{\frac{1}{4}}\|u\otimes v_r+w_r\otimes u\|_{L^4(0,\tau; L^3({\mathds{R}}^3)^{3\times 3})}
\\
&\le c'\,\tau^{\frac{1}{4}}\bigl(\|v_r\|_{L^\infty((0,\tau)\times{\mathds{R}}^3)^3)}+\|w_r\|_{L^\infty((0,\tau)\times{\mathds{R}}^3)^3)}\bigr)
\|u\|_{L^4(0,\tau; L^6({\mathds{R}}^3)^3)}.
\end{align*}
We first choose $r>0$,  such that $3rC_{4,6} \le \frac{\varepsilon}{2}$, next
$v_r$ and $w_r$ in $\C_c([0,T]\times \R^3)$,
satisfying~\eqref{eq:vr,wr} and last $\tau>0$ such that 
\[
c'\,\tau^{\frac{1}{4}}\bigl(\|v_r\|_{L^\infty((0,\tau)\times{\mathds{R}}^3)^3)}+\|w_r\|_{L^\infty((0,\tau)\times{\mathds{R}}^3)^3)}\bigr)
\le \frac{\varepsilon}{2}.
\]
This finally establishes~\eqref{eq:L4L6}.
\item[(ii)]
Let us now consider assertion~\eqref{eq:LpL3}. 
Let $1<p<\infty$. 
We slightly modify the expression of 
$B(\cdot,v)+B(w,\cdot)$ given by~\eqref{eq:B-R}:
\begin{equation}
\label{eq:B-R'}
B(u,v)+B(w,u)=\Delta Rf +(-\Delta)^{1/2}R\tilde{g},
\end{equation}
where we set
\begin{equation}
\label{eq:def-tildeg}
\tilde{g}(s)=-(-\Delta)^{-1/2}{\mathbb{P}}\bigl(\nabla\cdot
\bigl(u(s)\otimes v_r(s) +w_r(s)\otimes u(s)\bigr)\bigr), \quad s\in[0,T]
\end{equation}
The function $f$ defined by~\eqref{eq:def-f} is bounded in $L^p(0,\tau;L^3(\R^3)^3)$ by
$3r\|u\|_{L^p(0,\tau;L^3(\R^3)^3)}$, up to a multiplicative constant involving $C_{p,3}$ 
and the norm of bounded operator  
$(-\Delta)^{-1}{\mathbb{P}}\nabla\cdot$ from $L^{\frac{3}{2}}(\R^3)^{3\times 3}$ to $L^3(\R^3)^3$. 
The norm of $\tilde g$ in $L^p(0,\tau;L^3(\R^3)^3)$ 
is bounded by the norm of $u\otimes v_r+w_r\otimes u$ in $L^p(0,\tau; L^3(\R^3)^{3\times 3})$.
Indeed, the operator $(-\Delta)^{-\frac{1}{2}}{\mathbb{P}}\nabla \cdot$ is bounded from $L^3(\R^3)^{3\times 3}$ in $L^3(\R^3)^3$ and so
\begin{align*}
\|(-\Delta)^{1/2}R\tilde{g}\|_{L^p(0,\tau;L^3)}&\le 
c \|t\mapsto (-\Delta)^{1/2}e^{t\Delta}\|_{L^1(0,\tau;{\mathscr{L}}(L^3))}\|g\|_{L^p(0,\tau; L^3)}\\
&\lesssim \sqrt{\tau} \bigl(\|v_r\|_{L^\infty((0,\tau)\times{\mathds{R}}^3)^3}+\|w_r\|_{L^\infty((0,\tau)\times{\mathds{R}}^3)^3}\bigr)
\|u\|_{L^p(0,\tau; L^3)}.
\end{align*}
Proceeding as in item (i) settles~\eqref{eq:LpL3}.
\end{itemize}
This establishes~Proposition~\ref{prop:B}.
\end{proof}

\begin{remark}
Notice that if one assumes that $v$ and $w$ belong to the larger
space $L^\infty(0,T;L^3(\R^3)^3)$, and if $r>0$ is fixed, then in general one cannot
ensure the existence of $v_r$ and $w_r$ in $L^\infty((0,T)\times \R^3)$
such that $\esssup_{t\in(0,T)} \|v(t)-v_r(t)\|_{L^3}+\esssup_{t\in(0,T)}\|w(t)-w_r(t)\|_{L^3}<3r$.
This is the case if, for example, $v$ or $w$ are of the form $t^{-1}\phi(\cdot/t)$
with $\phi\in L^3(\R^3)$ and $r$ is small with respect to~$\|\phi\|_{L^3}$.
\end{remark}

\begin{proposition}\mbox{}
\label{prop:C}
\begin{enumerate}
\item
For all $\varepsilon>0$, there exists $r>0$ such that
all $v\in \C([0,T],L^3(\R^3)^3) +r\, L^\infty(0,T;L^3(\R^3)^3)$, 
there exists $\tau=\tau(\varepsilon,v)>0$ such that 
\[
C(v,\cdot):L^{\frac{4}{3}}(0,\tau;L^2({\mathds{R}}^3))\longrightarrow L^{\frac{4}{3}}(0,\tau;L^2({\mathds{R}}^3))\quad
\mbox{is bounded with norm less than }\varepsilon.
\]
\item 
For all $\varepsilon>0$ there exists $r>0$ such that, for all $v\in \C([0,T],L^3(\R^3)^3) +r\, L^\infty(0,T;L^3(\R^3)^3)$, 
there exists $\tau=\tau(\varepsilon,v)>0$ such that
\[
C(v,\cdot): L^2(0,\tau;L^{\frac{3}{2}}({\mathds{R}}^3))\cap
 L^{\frac{4}{3}}(0,\tau;L^2({\mathds{R}}^3)) 
 \longrightarrow
L^2(0,\tau;L^{\frac{3}{2}}({\mathds{R}}^3)) \quad \mbox{is bounded}
\]
with norm less than $\varepsilon$.
\item
For all $\varepsilon>0$ all $v\in L^4(0,T;L^6({\mathds{R}}^3))^3$, there exists
$\tau=\tau(\varepsilon,v)>0$ such that
\[
C(v,\cdot):L^2(0,\tau;L^{\frac{3}{2}}({\mathds{R}}^3))\longrightarrow L^2(0,\tau;L^{\frac{3}{2}}({\mathds{R}}^3))\quad
\mbox{is bounded with norm less than }\varepsilon.
\]
\item
For all $\varepsilon>0$, all $\vartheta\in L^2(0,T;L^{\frac{3}{2}}({\mathds{R}}^3))$, there exists $\tau=\tau(\varepsilon,\vartheta)>0$
such that
\[
C(\cdot,\vartheta):L^4(0,\tau;L^6({\mathds{R}}^3)^3)\longrightarrow L^{\frac{4}{3}}(0,\tau;L^2({\mathds{R}}^3))\quad
\mbox{is bounded with norm less than }\varepsilon.
\]
\end{enumerate}
\end{proposition}

\begin{proof}
We proceed as in the previous proposition. 
For $r>0$, let us choose $v_r\in \C_c([0,T]\times\mathds{R}^3)^3$ such that
\begin{equation}
\label{eq:supvr}
\esssup_{[0,T]}\|v-v_r\|_{L^3}\le 2r. 
\end{equation}
Then we have, for all $\theta\in L^{\frac{4}{3}}(0,\tau;L^2({\mathds{R}}^3))$,
\[
C(v,\theta)= C(v-v_r,\theta)+C(v_r,\theta)= \Delta R f + (-\Delta)^{1/2} Rg
\]
where, now, we set
$f(s)=(-\Delta)^{-1}{\rm div}\, \bigl(\theta(s)(v(s)-v_r(s))\bigr)$ 
and $g(s)= -(-\Delta)^{-1/2}{\rm div}\,\bigl(\theta(s)v_r(s)\bigr)$.
As in the proof of previous proposition we see that the norm of $f$
in $L^{\frac{4}{3}}(0,\tau ; L^2({\mathds{R}}^3))$
is bounded by the norm of $\bigl(\theta(v-v_r)\bigr)$ in 
$L^{\frac{4}{3}}(0,\tau ; L^{\frac{6}{5}}({\mathds{R}}^3)^3)$
(because of the Sobolev embedding $\dot W^{1,\frac{6}{5}}\hookrightarrow L^2$ in dimension~$3$).
As $\Delta R$ is a bounded operator in 
$L^{\frac{4}{3}}(0,\tau ; L^2({\mathds{R}}^3))$, we have
\begin{align*}
\|(-\Delta)Rf\|_{L^{\frac{4}{3}}(0,\tau ; L^2({\mathds{R}}^3))}
&\le C'_{\frac{4}{3},2}\, C\, \|\theta(v-v_r)\|_{L^{\frac{4}{3}}(0,\tau ; L^{\frac{6}{5}}({\mathds{R}}^3)^3)}\\
&\le 2rC'_{\frac{4}{3},2}\, C\,  \|\theta\|_{L^{\frac{4}{3}}(0,\tau ; L^2({\mathds{R}}^3))}
\end{align*}
where $C$ is the constant arising from the Sobolev embedding $\dot W^{1,\frac{6}{5}}\hookrightarrow L^2$ in 
dimension $3$ and $r$ comes from the choice of $v_r$.

The norm of $g$ in $L^{\frac{4}{3}}(0,\tau ; L^2({\mathds{R}}^3))$ is bounded by
the norm of 
$\theta \,v_r$ dans $L^{\frac{4}{3}}(0,\tau ; L^2({\mathds{R}}^3)^3)$,
because $(-\Delta)^{-\frac{1}{2}}{\rm div}\,$ is a bounded operator in
$L^2({\mathds{R}}^3)^3$.
Moreover, $\|(-\Delta)^{1/2}e^{t\Delta}\|_{\mathscr{L}(L^2(\R^3))}\lesssim t^{-1/2}$.
Then, viewing as before $(-\Delta)^{1/2}R$ as a convolution operator, we get 
\begin{align*}
\|(-\Delta)^{1/2}Rg\|_{L^{\frac{4}{3}}(0,\tau;L^2({\mathds{R}}^3))}&\le 
c\,\|t\mapsto (-\Delta)^{1/2}e^{t\Delta}\|_{L^1(0,\tau;{\mathscr{L}}(L^2({\mathds{R}}^3)))}\|g\|_{L^{\frac{4}{3}}(0,\tau; L^2({\mathds{R}}^3))}
\\
&\le c'\, \tau^{\frac{1}{2}}\|\theta\,v_r\|_{L^{\frac{4}{3}}(0,\tau; L^2({\mathds{R}}^3)^3)}
\\
&\le c'\,\tau^{\frac{1}{2}} \|v_r\|_{L^\infty((0,\tau)\times{\mathds{R}}^3)^3)}
\|\theta\|_{L^{\frac{4}{3}}(0,\tau; L^2({\mathds{R}}^3))}.
\end{align*}
We then choose $r>0$ such that $2rC'_{\frac{4}{3},2}\, C\le \frac{\varepsilon}{2}$, next 
$v_r\in {\mathscr{C}}_{\rm c}([0,T]\times{\mathds{R}}^3)^3$
satisfying  \eqref{eq:supvr} and last $\theta >0$ such that 
$c'\,\tau^{\frac{1}{2}} \|v_r\|_{L^\infty(((0,\tau)\times\mathds{R}^3)^3)}
\le\frac{\varepsilon}{2}$.
This establishes the first assertion of the proposition.

\medskip
To prove the second assertion we proceed as before: for $r>0$, we choose 
$v_r\in \C([0,T]\times\mathds{R}^3)^3$ such that~\eqref{eq:supvr}
holds. 
Then for all $\theta\in L^2(0,\tau;L^{\frac{3}{2}}({\mathds{R}}^3))$,
\begin{equation}
\label{eq:dec-C}
C(v,\theta)= C(v-v_r,\theta)+C(v_r,\theta)
= (-\Delta)^{\frac{3}{4}} R f + (-\Delta)^{1/2} Rg
\end{equation}
where $f(s)=-(-\Delta)^{-\frac{3}{4}}{\rm div}\, \bigl(\theta(s)(v(s)-v_r(s))\bigr)$ 
and $g(s)= -(-\Delta)^{-1/2}{\rm div}\,\bigl(\theta(s)v_r(s)\bigr)$.
We easily see that the norm of $f$ in $L^{\frac{4}{3}}(0,\tau ; L^{\frac{3}{2}}({\mathds{R}}^3))$ is controlled by the norm 
of $\bigl(\theta(v-v_r)\bigr)$ in $L^{\frac{4}{3}}(0,\tau ; L^{\frac{6}{5}}({\mathds{R}}^3)^3)$ (owing to the Sobolev embedding 
$\dot W^{\frac{1}{2},\frac{6}{5}}\hookrightarrow L^{\frac{3}{2}}$ in dimension~$3$).
As $\bigl(\frac{{\rm d}}{{\rm d}t}\bigr)^{\frac{1}{4}}(-\Delta)^{\frac{3}{4}}R$ is a bounded operator in $L^{\frac{4}{3}}(0,\tau ; L^{\frac{3}{2}}({\mathds{R}}^3))$, 
we have
\begin{equation}
\begin{split}
\label{eq:comp-Rf}
\|(-\Delta)^{\frac{3}{4}}Rf\|_{L^2(0,\tau ; L^{\frac{3}{2}}({\mathds{R}}^3))}
&\le \tilde{C}\bigl\|\bigl(\tfrac{{\rm d}}{{\rm d}t}\bigr)^{\frac{1}{4}}(-\Delta)^{\frac{3}{4}}Rf\bigl\|_{L^{\frac{4}{3}}(0,\tau ; L^{\frac{3}{2}})}\\
&\le \tilde{C}C_{\frac{4}{3},\frac{3}{2}}\, C\, \|\theta(v-v_r)\|_{L^{\frac{4}{3}}(0,\tau ; L^{\frac{6}{5}}({\mathds{R}}^3)^3)}\\
&\le \tilde{C}C_{\frac{4}{3},\frac{3}{2}}\, C\, (r+\eta) \|\theta\|_{L^{\frac{4}{3}}(0,\tau ; L^2({\mathds{R}}^3))}.
\end{split}
\end{equation}
Here $\tilde{C}$ is the constant coming from the Sobolev embedding $\dot W^{\frac{1}{4},\frac{4}{3}}\hookrightarrow L^2$ 
in dimension~$1$,
$C$ is the constant of the Sobolev embedding $\dot W^{\frac12,\frac{6}{5}}\hookrightarrow L^\frac32$ in dimension~$3$ and 
$r$ comes from the choice of $v_r$.
The norm of $g$ in $L^2(0,\tau ; L^{\frac{3}{2}}({\mathds{R}}^3))$ is controlled by the norm of $\theta \,v_r$ in 
$L^2(0,\tau ; L^{\frac{3}{2}}({\mathds{R}}^3)^3)$, because
$(-\Delta)^{-\frac{1}{2}}{\rm div}\,$ is a bounded operator from 
$L^{\frac{3}{2}}({\mathds{R}}^3)^3$ to $L^{\frac{3}{2}}({\mathds{R}}^3)$. 
As $(-\Delta)^{\frac12}R$ is a convolution operator, we can write 
\begin{equation}
\begin{split}
\label{eq:comp-Rg}
\|(-\Delta)^{\frac12}Rg\|_{L^2(0,\tau;L^{\frac{3}{2}}({\mathds{R}}^3))}&\le 
c\,\|t\mapsto (-\Delta)^{\frac12}e^{t\Delta}\|_{L^1(0,\tau;{\mathscr{L}}(L^{\frac{3}{2}}({\mathds{R}}^3)))}
\|g\|_{L^2(0,\tau; L^{\frac{3}{2}}({\mathds{R}}^3))}
\\
&\le c'\, \tau^{\frac{1}{2}}\|\theta\,v_r\|_{L^2(0,\tau; L^{\frac{3}{2}}({\mathds{R}}^3)^3)}
\\
&\le c'\,\tau^{\frac{1}{2}} \|v_r\|_{L^\infty((0,\tau)\times{\mathds{R}}^3)^3)}
\|\theta\|_{L^2(0,\tau; L^{\frac{3}{2}}({\mathds{R}}^3))}.
\end{split}
\end{equation}
It just remains to choose $r,\eta>0$ such that 
$2r\tilde{C} C_{\frac{4}{3},\frac{3}{2}}\, C\le \frac{\varepsilon}{2}$, next 
$v_r\in L^\infty((0,T)\times{\mathds{R}}^3)^3$
such that~\eqref{eq:supvr} holds and finally $\theta>0$ such that
$c'\,\tau^{\frac{1}{2}} \|v_r\|_{L^\infty((0,\tau)\times{\mathds{R}}^3)^3)}
\le\frac{\varepsilon}{2}$. This establishes the second assertion of the proposition.

\medskip
Let us prove the third assertion.
As $v\in L^4(0,T;L^6(\R^3)^3)$, for an arbitrary 
$r>0$ we can choose now $v_r\in L^\infty((0,T)\times{\mathds{R}}^3)^3$ such that
\begin{equation}
\label{eq:new-r}
\|v-v_r\|_{L^4(0,T;L^6(\R^3)^3)}<r.
\end{equation} 
If $\tau>0$ and $\theta\in L^2(0,\tau;L^\frac32(\R^3))$, then
$(v-v_r)\theta\in L^\frac43(0,\tau;L^\frac65(\R^3)^3)$ by H\"older inequality.
Therefore, splitting~$C(v,\theta)$ as in~\eqref{eq:dec-C}, 
the above computations~\eqref{eq:comp-Rf}-\eqref{eq:comp-Rg}
can be reproduced: the only change that needs to be done is  
the application of~\eqref{eq:new-r} instead of~\eqref{eq:supvr}.
We get in this way
\begin{align*}
\|(-\Delta)^{\frac{3}{4}}Rf\|_{L^2(0,\tau ; L^{\frac{3}{2}}({\mathds{R}}^3))}
&\le \tilde{C}C_{\frac{4}{3},\frac{3}{2}}\, C\, r 
 \|\theta\|_{L^2(0,\tau ; L^\frac32({\mathds{R}}^3))}.
\end{align*}
This, combined with~\eqref{eq:comp-Rg} proves our third assertion.

\medskip

The proof of the fourth assertion follows the same scheme.
For $r>0$, choose 
$\vartheta_r\in \mathscr{C}_c([0,T]\times\mathds{R}^3)$ such that 
\begin{equation}
\label{eq:thetar}
\|\vartheta-\vartheta_r\|_{L^2(0,T;L^{\frac{3}{2}}({\mathds{R}}^3))}\le r.
\end{equation}
Then, for all $v\in L^4(0,\tau;L^6({\mathds{R}}^3)^3)$, we have 
\begin{equation*}
C(v,\vartheta)=C(v,\vartheta-\vartheta_r)+C(v,\vartheta_r)=\Delta Rf+(-\Delta)^{1/2} Rg
\end{equation*}
with $f(s)=(-\Delta)^{-1}{\rm div}\,\bigl((\vartheta(s)-\vartheta_r(s))v(s)\bigr)$ 
and $g(s)=-(-\Delta)^{-\frac{1}{2}}{\rm div}\,\bigl(\vartheta_r(s)v(s)\bigr)$. 
One easily shows that the norm of~$f$ in
$L^{\frac{4}{3}}(0,\tau; L^2({\mathds{R}}^3))$ is bounded by
$C r\|v\|_{L^4(0,\tau;L^6({\mathds{R}}^3)^3)}$, where
$C$ is the norm of $(-\Delta)^{-1}{\rm div}\,:L^{6/5}({\mathds{R}}^3)^3\to L^2({\mathds{R}}^3)$. 
Thus, as the operateur
$\Delta R$ is bounded on $L^{\frac{4}{3}}(0,\tau; L^2({\mathds{R}}^3))$ by $C'_{\frac{4}{3},2}=C\,C_{\frac43,2}$, we have 
\[
\|\Delta Rf\|_{L^{\frac{4}{3}}(0,\tau; L^2({\mathds{R}}^3))}\le C'_{\frac{4}{3},2} r\|v\|_{L^4(0,\tau;L^6({\mathds{R}}^3)^3)}.
\]
The norm of $g$ in $L^{\frac{4}{3}}(0,\tau; L^2({\mathds{R}}^3))$ is bounded by
the norm of $\vartheta_r$ in $L^2(0,\tau;L^3({\mathds{R}}^3))$ and the norm of
$v$ in $L^4(0,\tau;L^6({\mathds{R}}^3)^3)$. 
As $(-\Delta)^{1/2} R$ is a convolution operator, we deduce that
\begin{align*}
\|(-\Delta)^{\frac{1}{2}}Rg\|_{L^{\frac{4}{3}}(0,\tau; L^2({\mathds{R}}^3))}&\le
\|t\mapsto (-\Delta)^{1/2}e^{t\Delta}\|_{L^1(0,\tau;{\mathscr{L}}(L^2({\mathds{R}}^3)))}\|g\|_{L^{\frac{4}{3}}(0,\tau; L^2({\mathds{R}}^3))}\\
&\le c\,\tau^{\frac{1}{2}} \|\vartheta_r\|_{L^2(0,\tau;L^3({\mathds{R}}^3))}\|v\|_{L^4(0,\tau;L^6({\mathds{R}}^3)^3)}.
\end{align*}
We then choose $r>0$ such that $C'_{\frac{4}{3},2} r\le \frac{\varepsilon}{2}$, 
next $\vartheta_r\in {\mathscr{C}}_{\rm c}([0,T]\times\mathds{R}^3)$
satisfying~\eqref{eq:thetar} and last $\tau>0$ such that 
$c\,\tau^{\frac{1}{2}} \|\vartheta_r\|_{L^2(0,\tau;L^3({\mathds{R}}^3))}\le \frac{\varepsilon}{2}$.
We thus get the last assertion of the proposition.
\end{proof}

\begin{proposition}
\label{prop:L}
For all $\tau>0$, the operator $L$ defined by \eqref{def:L} is linear and bounded
from  
$L^2(0,\tau;L^{\frac{3}{2}}({\mathds{R}}^3))$ to $L^4(0,\tau; L^6({\mathds{R}}^3)^3$ and from $L^{\frac{4}{3}}(0,\tau;L^2({\mathds{R}}^3))$ to 
$L^4(0,\tau; L^6({\mathds{R}}^3)^3)$, with operator norms independent on~$\tau$. 
Moreover, for all $p\in[1,\infty)$, $L$ is bounded from 
$L^2(0,\tau;L^{\frac32}(\R^3))$ to $L^p(0,\tau;L^3(\R^3))$, 
with norm of order $\tau^{1/p}$.
\end{proposition}

\begin{proof}\mbox{}
For $\theta \in L^2(0,\tau;L^{\frac{3}{2}}({\mathds{R}}^3))$, we write
\[
L(\theta)=
\bigr(\tfrac{{\rm d}}{{\rm d}t}\bigr)^{-\frac{1}{4}} \Bigl(\bigr(\tfrac{{\rm d}}{{\rm d}t}\bigr)^{\frac{1}{4}}(-\Delta)^{\frac{3}{4}} R \varphi \Bigr),
\]
where $\varphi(s)=(-\Delta)^{-\frac{3}{4}}{\mathbb{P}}\bigl(\theta(s)e_3\bigr)$.
Observe that $\varphi\in L^2(0,\tau;L^6({\mathds{R}}^3)^3)$, because of the Sobolev embedding 
$(-\Delta)^{-\frac{3}{4}}(L^{\frac{3}{2}})\hookrightarrow L^6$ (in dimension~3), with norm bounded by the norm of $\theta$
in $L^2(0,\tau;L^{\frac{3}{2}}({\mathds{R}}^3))$. By Theorem~\ref{thm:regmax},
we deduce that 
$L(\theta)\in \bigr(\frac{{\rm d}}{{\rm d}t}\bigr)^{-\frac{1}{4}} \Bigl(L^2(0,\tau;L^6({\mathds{R}}^3)^3)\Bigr)
\hookrightarrow L^4(0,\tau;L^6({\mathds{R}}^3)^3)$, 
the last inclusion arising from the Sobolev embedding
$\bigr(\frac{{\rm d}}{{\rm d}t}\bigr)^{-\frac{1}{4}}(L^2)\hookrightarrow L^4$ (in
dimension~1).
This establishes the first assertion of the proposition.

When $\theta\in L^{\frac{4}{3}}(0,\tau;L^2({\mathds{R}}^3))$, we write
\begin{equation}
\label{eq:L-theta}
L(\theta)=
\bigr(\tfrac{{\rm d}}{{\rm d}t}\bigr)^{-\frac{1}{2}} \Bigl(\bigr(\tfrac{{\rm d}}{{\rm d}t}\bigr)^{\frac{1}{2}}(-\Delta)^{\frac{1}{2}} R \psi \Bigr),
\end{equation}
with $\psi(s)=(-\Delta)^{-\frac{1}{2}}{\mathbb{P}}\bigl(\theta(s)e_3\bigr)$.
Notice that $\psi\in L^{\frac{4}{3}}(0,\tau;L^6({\mathds{R}}^3)^3)$, because of
the Sobolev embedding $(-\Delta)^{-\frac{1}{2}}L^2\hookrightarrow L^6$ (in dimension~ 3), with norm 
bounded by the norm of $\theta$ in 
$L^{\frac{4}{3}}(0,\tau;L^2({\mathds{R}}^3))$. 
Applying Theorem~\ref{thm:regmax}
with $\alpha=\frac{1}{2}$, we get 
$L(\theta)\in \bigr(\frac{{\rm d}}{{\rm d}t}\bigr)^{-\frac{1}{2}} \Bigl(L^{\frac{4}{3}}(0,\tau;L^6({\mathds{R}}^3)^3)\Bigr)
\hookrightarrow L^4(0,\tau;L^6({\mathds{R}}^3)^3)$.
The last inclusion comes from the Sobolev embedding
$\bigr(\frac{{\rm d}}{{\rm d}t}\bigr)^{-\frac{1}{2}}(L^{\frac{4}{3}})\hookrightarrow L^4$ (in dimension 1).
The second assertion of the proposition follows.

Next, for $\theta\in L^2(0,\tau;L^{\frac{3}{2}}({\mathds{R}}^3))$, let us write $L(\theta)$ as before in~\eqref{eq:L-theta}.
By Sobolev embedding $(-\Delta)^{-\frac{1}{2}}(L^{\frac{3}{2}})\hookrightarrow L^3$ in dimension~3, we have 
$\psi\in L^2(0,\tau;L^3)$ with norm bounded by $\|\theta\|_{L^2(0,\tau;L^{\frac{3}{2}})}$. 
By Theorem~\ref{thm:regmax} with $\alpha=\frac{1}{2}$, we deduce that 
$L(\theta)\in \bigr(\tfrac{{\rm d}}{{\rm d}t}\bigr)^{-\frac{1}{2}}\bigl(L^2(0,\tau;L^3)\bigr)\hookrightarrow L^p(0,\tau;L^3)$
for all $1\le p<\infty$. The last inclusion follows, for $2<p<\infty$, from 
the H\"older injection $L^2((0,\tau))\hookrightarrow  L^q(0,\tau)$ for $q\in [1,2]$
(with norm $\tau^{1/q-1/2}$) and Hardy-Littlewood-Sobolev inequality
$\bigr(\frac{{\rm d}}{{\rm d}t}\bigr)^{-\frac{1}{2}}(L^q)\hookrightarrow L^p$ (in dimension 1), for all $p\in(2,\infty)$ and 
$\frac1p=\frac1q-\frac12$.
For $1\le p\le 2$ it is sufficient to apply once more H\"older inequality.
\end{proof}

\section{The proof of the uniqueness}

We need a few lemmas before proving Theorem~\ref{th:uni}.

\begin{lemma}
\label{lem:S'}
Let $(u_0,\theta_0)\in \mathscr{S}'(\R^3)^3\times \mathscr{S}'(\R^3)$ with ${\rm div\,}u_0=0$, and let 
$(u,\theta)\in L^\infty(0,T;L^3(\R^3)^3)\times L^2(0,T;L^{\frac32}(\R^3))$ be a mild solution of~\eqref{def:sol} with
initial data $(u_0,\theta_0)$.
Then 
\[
(u,\theta)\in \mathscr{C}([0,T],\mathscr{S'}(\R^3)^3\times\mathscr{S'}(\R^3)).
\] 
Moreover, we have $u_0\in L^3(\R^3)^3$ and for every $t\in[0,T]$, $u(t)$ does also belong to $L^3(\R^3)^3$.
\end{lemma}

\begin{proof}
Let us denote by $F(t,x)$ the kernel of the operator $e^{t\Delta}\P\nabla\cdot$. 
It is well known, and easy to check, that $F$ satisfies the scaling
relations $F(t,x)=t^{-\frac32}F(1,x/\sqrt t)$, with $F(1,\cdot)\in (L^1(\R^3)\cap \mathscr{C}_0(\R^3))^{3\times 3}$.
From these properties and the dominated convergence theorem one deduces that,
for all $1\le p\le \infty$, that $F\in \mathscr{C}(0,\infty;L^p(\R^3))$.
Moreover, $\|F(t,\cdot)\|_1=t^{-\frac12}\|F(1,\cdot)\|_1$.

Now, if $(u,\theta)\in X_{T,r}$, then 
$u\otimes u\in L^\infty(0,T;L^{3/2}(\R^3)^{3\times 3})$.
Then, recalling the definition of the bilinear operator~$B$ and applying the above properties of $F$ with $p=1$, 
next applying the $L^1$-$L^{3/2}$ convolution inequality, shows that the map $t\mapsto B(u,u)(t)$ is continuous 
from $(0,T]$ to $L^{3/2}(\R^3)^3$.
Moreover, $\|B(u,u)(t)\|_{L^{3/2}}\to0$ as $t\to0$.
Hence, the map $t\mapsto B(u,u)(t)$ is continuous from $[0,T]$ to $L^{3/2}(\R^3)^3$
with value~$0$ at $t=0$.

Let us now consider $L(\theta)$.
Using the fact that the heat kernel $e^{-|x|^2/(4t)}/(4\pi\,t)^{3/2}$ is in
$\mathscr{C}_b(0,\infty ;L^1(\R^3))$, we readily see that 
$L(\theta)\in \mathscr{C}((0,T];L^{3/2}(\R^3))$.
To study the behavior of $L(\theta)$ near $t=0$ we consider 
$\varphi\in\mathscr{S}(\R^3)$ and  observe,
computing the Fourier transform of $\P\theta e_3$ with respect to the space variable,
 that
$t\mapsto \widehat h(t,\cdot)=\widehat{\P\theta e_3}(t,\cdot)$ belongs to $L^2(0,T;L^3(\R^3))$ by 
the Hausdorff-Young theorem. Then we have
\[
\begin{split}
|\langle L(\theta)(t),\varphi\rangle|
 &\le \int_0^t |\langle \widehat h(s),e^{-(t-s)|\cdot|^2}\widehat\varphi\rangle|\,{\rm d}s
\le \int_0^t \|\widehat h(s)\|_{L^3}\|\widehat\varphi\|_{L^{3/2}}\,{\rm d}s\\
&\le \|\widehat \varphi\|_{L^{3/2}}  \int_0^t\|\theta(s)\|_{L^{3/2}}\,{\rm d}s
\le C_\varphi\|\theta\|_{L^2(0,T;L^{3/2}(\R^3))} \,\sqrt t.
\end{split}
\]
Therefore, $L(\theta)(t)\to0$ as $t\to0$ in $\mathscr{S}'(\R^3)$
and we deduce that $L(\theta)\in \mathscr{C}([0,T],\mathscr{S}'(\R^3)^3)$,
with value $0$ at $t=0$.

Let us now consider $C(u,\theta)$.
We have $u\theta\in L^2(0,T;L^1(\R^3)^3)$. Moreover, the kernel of the operator
$e^{t\Delta}\nabla\cdot$ has the same scaling properties as $F$.
Therefore, proceeding as for $B(u,u)$ we see on the one hand that 
$C(u,\theta)\in \mathscr{C}((0,T];L^1(\R^3)^3)$.
On the other hand, we can also write 
\[
C(u,\theta)(t)={\rm div}\int_0^t e^{(t-s)\Delta}(u\theta){\,\rm d}s.
\]
But the $L^1(\R^3)$-norm of $\int_0^t e^{(t-s)\Delta}(u\theta){\,\rm d}s$
is bounded by $\sqrt t\|u\theta\|_{L^2(0,T;L^1(\R^3))}$ that goes to zero as $t\to0$.
Hence, $C(u,\theta)(t)\xrightarrow[t\to 0]{}0$ in $\mathscr{S}'(\R^3)^3$, by the continuity of the divergence 
operator from $L^1$ to $\mathscr{S}'$.

For the linear terms $a$ and $b$ it is obvious that they are 
both in $\mathscr{C}([0,T],\mathscr{S}'(\R^3))$, with values at $t=0$
given by $u_0$ and $\theta_0$, respectively.

Summarising, from the equation~\eqref{def:sol} we see that
$(u,\theta)\in \mathscr{C}([0,T],\mathscr{S}'(\R^3)^3\times\mathscr{S}'(\R^3))$, with values at $t=0$ given by $(u_0,\theta_0)$.
But $u\in L^\infty(0,T;L^3(\R^3)^3)$, hence, for all $0\le t\le T$, we can find
a sequence $t_n\xrightarrow[n\to\infty]{} t$, contained in $[0,T]$, such that $u(t_n)\in L^3(\R^3)^3$
for all $n\in{\mathds{N}}$,
with $L^3$-norm uniformly bounded by $ \|u\|_{L^\infty(0,T;L^3(\R^3)^3)}$, and
$u(t_n)\xrightarrow[n\to\infty]{} u(t)$ in $\mathscr{S}'(\R^3)^3$. By duality we deduce that $u(t)\in L^3(\R^3)$ 
for every $t\in[0,T]$.
In particular, the initial velocity $u_0$ must belong to $L^3(\R^3)$.
\end{proof}

\begin{lemma}
\label{lem:uL4L6}
There exists an absolute constant $r_0>0$ such that if $0\le r<r_0$ and 
$(u,\theta)\in X_{T,r}$ is a solution of~\eqref{def:sol}, with 
$(u_0,\theta_0)\in \mathscr{S}'(\R^3)^3\times \mathscr{S}'(\R^3)$ and
${\rm div\,}u_0=0$,
then there exists $\tau>0$ such that $u\in L^4(0,\tau;L^6(\R^3)^3)$.
\end{lemma}

\begin{proof}
Let us take $p=2$ throughout this proof (any other choice $1<p<\infty$ would do: a different choice of~$p$ would just 
affect the value of $r_0$ and $\tau$). 
We know, by Proposition~\ref{prop:B}, that there exists $r_0$ and $\tau>0$ such that
if $(u,\theta)\in X_{T,r}$ with $0\le r<r_0$, then the norm
of the linear operator
$B(\cdot,u)$ from $L^4(0,\tau;L^6(\R^3)^3)\cap L^p(0,\tau;L^3(\R^3)^3)$ to itself
is bounded, with norm smaller than $\frac{1}{2}$. 
This shows that ${\rm Id}-B(\cdot,u)$ is invertible
in $L^4(0,\tau;L^6(\R^3)^3)\cap L^p(0,\tau;L^3(\R^3)^3)$.

Moreover, as $\theta\in L^2(0,\tau;L^{\frac{3}{2}}(\R^3))$ by our assumption,
we get from Proposition~\ref{prop:L} that
$L(\theta)\in L^4(0,\tau;L^6(\R^3)^3)\cap L^p(0,\tau;L^3(\R^3)^3)$. 

As observed in the previous Lemma, we have $u_0\in L^3(\R^3)$.
Moreover, $L^3(\R^3)\subset \dot B^0_{3,3}(\R^3)\subset 
\dot B^{-1/2}_{6,3}(\R^3)\subset \dot B^{-1/2}_{6,4}(\R^3) $.
See~\cite[Chapt.~2]{BaCD} for generalities on Besov spaces. The characterisation
of Besov spaces through the heat kernel (see \cite[Theorem~2.34]{BaCD}) then implies
that  $t\mapsto e^{t\Delta}u_0\in L^4(0,\tau;L^6(\R^3)^3)$. 
Since we have also that $t\mapsto e^{t\Delta}u_0 \in {\mathscr{C}}([0,\tau];L^3(\R^3)^3)$, we obtain
\[
a\in L^4(0,\tau;L^6(\R^3)^3)\cap L^p(0,\tau;L^3(\R^3)^3),
\qquad \text{for all $1<p<\infty$}.
\]
These considerations allow us to define 
\[
\tilde{u}=\bigl({\rm Id}-B(\cdot,u)\bigr)^{-1}\bigl(a+L(\theta)\bigr).
\]
We would like to show that $u=\tilde{u}$.
By the assumption on $u$ and the construction of $\tilde{u}$,
these two functions satisfy 
\[
u=B(u,u)+a+L(\theta)\quad\mbox{and}\quad \tilde{u}=B(\tilde{u},u)+a+L(\theta).
\]
Moreover, $\tilde u\in  L^4(0,\tau;L^6(\R^3)^3)\cap L^p(0,\tau;L^3(\R^3)^3)$.
Their difference $v:=u-\tilde{u}$ satisfies 
\[
v\in L^p(0,\tau ; L^3(\R^3)^3)\quad \mbox{and}\quad v=B(v,u).
\]
Reducing (if necessary) the value of $\tau$, we deduce from 
the last point of Proposition~\ref{prop:B} that
\[
\|v\|_{L^p(0,\tau,L^3)}\le \tfrac{1}{2}\|v\|_{L^p(0,\tau;L^3(\R^3)^3)}.
\]
This implies that $v=0$ in $L^p(0,\tau;L^3(\R^3)^3)$.
In particular, $u=\tilde u\in  L^4(0,\tau;L^6(\R^3)^3)$.
\end{proof}

\begin{remark}
Under the conditions of Lemma~\ref{lem:S'}, the initial temperature $\theta_0$ must belong to the inhomogeneous 
Besov space $B^{-1}_{3/2,2}(\R^3)$.
Indeed, $\theta\in L^2(0,\tau;L^{\frac32}(\R^3))$,
and $C(u,\theta)$ then belongs to this same space by the third claim of
Proposition~\ref{prop:C} and the previous lemma,
for $\tau>0$ small enough. 
Then, by the second equation of~\eqref{def:sol}, we obtain
$b\in L^2(0,\tau;L^{\frac32}(\R^3))$. The characterisation of inhomogeneous Besov spaces
with negative  regularity (see \cite[Theorem~5.3]{Lem02}) then immediately gives $\theta_0\in B^{-1}_{3/2,2}(\R^3)$.
\end{remark}

\begin{lemma}
\label{lem:theta-reg}
Let $0\le r<r_0$ and $(u_1,\theta_1)$ and $(u_2,\theta_2)$ be two mild solutions of~\eqref{B}
in $X_{T,r}$ arising from $(u_0,\theta_0)\in \mathscr{S}'(\R^3)^3\times\mathscr{S}'(\R^3)$,
with ${\rm div}\,u_0=0$. Let also
$\theta=\theta_1-\theta_2$.
Then there exists $\tau>0$ such that $\theta\in L^{\frac{4}{3}}(0,\tau;L^2(\R^3))$.
\end{lemma}

\begin{proof}
Let $u=u_1-u_2$. Then
$(u,\theta)\in X_{T,r}$ satisfies \eqref{eq:unicite}.
By~Lemma~\ref{lem:uL4L6} we know that there exists $\tau_0>0$ such that
$u_1,u_2\in L^4(0,\tau_0;L^6(\R^3)^3)$.
Applying the last two assertions of Proposition~\ref{prop:C} we get
$C(u,\theta_2)\in L^{\frac{4}{3}}(0,\tau_0;L^2(\R^3))
\cap L^2(0,\tau_0 ;L^{\frac{3}{2}}(\R^3))$.
The first and the second assertions of Proposition~\ref{prop:C} ensure the existence
of $\tau >0$ (we can assume $\tau\le\tau_0$) such that
$C(u_1,\cdot)$ is bounded from $L^{\frac{4}{3}}(0,\tau;L^2(\R^3))\cap L^2(0,\tau;L^{\frac{3}{2}}(\R^3))$ to itself, 
with norm less than $\frac{1}{2}$.
Therefore we can define 
\[
\tilde{\theta}=\bigl({\rm Id}-C(u_1,\cdot)\bigr)^{-1}(C(u,\theta_2)).
\]
We see that $\tilde{\theta}\in L^{\frac{4}{3}}(0,\tau;L^2(\R^3))\cap L^2(0,\tau ;L^{\frac{3}{2}}(\R^3))$, and moreover
$\tilde{\theta}=C(u_1,\tilde{\theta})+C(u,\theta_2)$. 
Let $\psi=\theta-\tilde{\theta}$. 
Then, subtracting the second equation in~\eqref{eq:unicite}, we obtain
\[
\psi\in L^2(0,\tau;L^{\frac{3}{2}}(\R^3))\quad\mbox{ and } \quad \psi=C(u_1,\psi).
\]
But $u_1\in L^4(0,\tau;L^6(\R^3)^3)$ by Lemma~\ref{lem:uL4L6}. Hence, 
applying the third assertion of Proposition~\ref{prop:C} we get $\psi=0$ and so 
$\theta=\tilde{\theta}$. The latter equality implies that 
$\theta\in L^{\frac{4}{3}}(0,\tau;L^2(\R^3))$.
\end{proof}

\begin{proof}[Proof of Theorem~\ref{th:uni}.]
Let $r_0>0$ be the absolute constant determined in Lemma~\ref{lem:uL4L6}.
Assume that $(u_1,\theta_1)$ and
$(u_2,\theta_2)$ are two mild solutions in $X_{T,r}$ of~\eqref{B}, with $0\le r<r_0$, starting from  
$(u_0,\theta_0)\in \mathscr{S}'(\R^3)^3\times \mathscr{S}'(\R^3)$. 
In fact, by Lemma~\ref{lem:S'}, there is no restriction in assuming that $u_0\in L^3(\R^3)^3$.

Then, setting $u=u_1-u_2$  and $\theta=\theta_1-\theta_2$, the couple
$(u,\theta)$ satisfies~\eqref{eq:unicite}.
As $\theta_1,\theta_2\in L^2(0,T;L^{\frac{3}{2}}({\mathds{R}}^3))$ by our assumption and the first item of~Proposition~\ref{prop:L}, we know that
$L(\theta)\in L^4(0,T; L^6({\mathds{R}}^3)^3)$.
Applying~Proposition~\ref{prop:B}, we know that there exists $\tau>0$ such that
$\|B(u,u_1)+B(u_2,u)\|_{L^4(0,\tau; L^6({\mathds{R}}^3)^3)}\le \frac{1}{2} \|u\|_{L^4(0,\tau; L^6({\mathds{R}}^3)^3)}$.
This allows us to show, applying Lemma~\ref{lem:uL4L6}, next using the first equation in~\eqref{eq:unicite}, 
that 
\[
\|u\|_{L^4(0,\tau; L^6({\mathds{R}}^3)^3)}\le 2\|L(\theta)\|_{L^4(0,\tau; L^6({\mathds{R}}^3)^3)}.
\]
Applying the first assertion of Proposition~\ref{prop:C}, with $v=u_1$
and the last assertion of  Proposition~\ref{prop:C} with $\vartheta=\theta_2$,
we deduce from the second equality in~\eqref{eq:unicite} that, for all
$\varepsilon>0$, there exists $0<\tau'\le\tau$ such that 
\[
\|\theta\|_{L^{\frac{4}{3}}(0,\tau';L^2({\mathds{R}}^3))}\le 
\varepsilon \bigl(\|\theta\|_{L^{\frac{4}{3}}(0,\tau';L^2({\mathds{R}}^3))} + \|u\|_{L^4(0,\tau'; L^6({\mathds{R}}^3)^3)}\bigr).
\]
But the ${L^{\frac{4}{3}}(0,\tau';L^2({\mathds{R}}^3))}$-norm of $\theta$ is finite
by Lemma~\ref{lem:theta-reg}, so
\[
\|\theta\|_{L^{\frac{4}{3}}(0,\tau';L^2({\mathds{R}}^3))}\le \frac{\varepsilon}{1-\varepsilon}\,\|u\|_{L^4(0,\tau'; L^6({\mathds{R}}^3)^3)}.
\]
The second item of Proposition~\ref{prop:L} allows us to take 
$\varepsilon>0$ such that
\[
2\,\frac{\varepsilon}{1-\varepsilon}\,\|L\|_{{\mathscr{L}}(L^{4/3}(0,\tau';L^2({\mathds{R}}^3)),L^4(0,\tau'; L^6({\mathds{R}}^3)^3)}<1.
\]
We conclude that $u=0$ in $L^4(0,\tau'; L^6({\mathds{R}}^3)^3)$. 
This implies that $\theta=0$ \emph{a.e.} on $(0,\tau')$.
The uniqueness is thus established at least during a short time interval $[0,\tau')$,
for a suitable $0<\tau'\le T$. 

A standard argument now allows us to deduce that the uniqueness holds, in fact, in the whole interval $[0,T]$: 
let $\tau^*$ be the supremum of the times
$t_0\in[0,T]$ such that $(u_1,\theta_1)=(u_2,\theta_2)$ in $X_{t_0,r}$.
Let us show that $\tau^*=T$. Indeed, otherwise, by the
continuity of $(u_1,\theta_1)$ and $(u_2,\theta_2)$ from  $[0,T]$ to 
$\mathscr{C}([0,T],\mathscr{S}'(\R^3)^3\times \mathscr{S}'(\R^3))$,
established in Lemma~\ref{lem:S'},
we deduce that 
\begin{equation}
\label{eq:shift}
(u_1(\tau^*),\theta_1(\tau^*))=(u_2(\tau^*),\theta_2(\tau^*))
\in \mathscr{S}'(\R^3)^3\times \mathscr{S}'(\R^3).
\end{equation}
But $(u_1,\theta_1)(\cdot+\tau^*)$ and $(u_2,\theta_2)(\cdot+\tau^*)$
are mild solutions of~\eqref{B} in $X_{T-\tau^*,r}$, with initial data given 
by~\eqref{eq:shift}. Therefore, applying the uniqueness
result in short-time intervals established before, we see that there exists $\tau''$, 
such that $0<\tau''<T-\tau^*$, and 
$(u_1,\theta_1)(\cdot+\tau^*)=(u_2,\theta_2)(\cdot+\tau^*)$
in $X_{\tau'',r}$.
Then $(u_1,\theta_1)=(u_2,\theta_2)$ in $X_{\tau^*+\tau'',r}$ and this would contradict 
the definition of $\tau^*$. The uniqueness is thus granted in the whole
interval $[0,T]$.
\end{proof}

\section{Existence}
\label{sec:existence}

Let us prove Theorem~\ref{th:existence}, that ensures the existence of solution in the space where we obtained the uniqueness.
In fact, an existence theorem for solutions to the Boussinesq system was established in~\cite{BraH}, under assumptions more 
general than that of Theorem~\ref{th:existence}. However, the solution constructed in~\cite{BraH}
a priori does not satisfy the required condition on the temperature, $\theta\in  L^2(0,T;L^{\frac32}(\R^3))$. Therefore, what remains 
to do in order to establish Theorem~\ref{th:existence}, is to show that the solution constructed in~\cite{BraH}
does satisfy such condition, as soon as the initial temperature does belong to $B^{-1}_{3/2,2}(\R^3)$.

For this, let us introduce some useful function spaces:
For $1\le p\le \infty$ and $0<T\le\infty$, we define  $Z_{p,T}$ to be
the subspace of all vector fields
$u\in L^1_{\rm loc}(0,T;L^p(\R^3)^3)$ such that
\[
\|u\|_{Z_{p,T}} =\esssup_{t\in(0,T)}t^{\frac12(1-\frac3p)}\|u(t)\|_p<\infty.
\]
In the same way, let $Y_{q,T}$ be 
the subspace of all the functions 
$\theta\in L^1_{\rm loc}(0,T;L^q(\R^3))$ such that
\[
\|\theta\|_{Y_{q,T}} =\esssup_{t\in(0,T)}t^{\frac32(1-\frac1q)}\|\theta(t)\|_q<\infty.
\]
We will need the following bilinear estimate:

\begin{proposition}
\label{prop:lo}
For all $u\in Z_{6,T}$ and $\theta\in L^2(0,T;L^\frac32(\R^3))$,
\[
\|C(u,\theta)\|_{L^2(0,T;L^\frac32(\R^3))} 
\le \kappa \|u\|_{Z_{6,T}}\|\theta\|_{L^2(0,T;L^\frac32(\R^3))},
\]
where  $\kappa>0$ is some constant independent on $T$, $u$ and $\theta$.
\end{proposition}

\begin{proof}
Using that $\|u(s)\|_{L^6}\le s^{-1/4}\|u\|_{Z_{6,T}}$ and letting 
$f(s)=s^{-1/4}\|\theta(s)\|_{L^{3/2}}{\mathds{1}}_{\R^+}(s)$ we can estimate
\[
\|C(u,\theta)(t)\|_{L^{3/2}}
\le c \|u\|_{Z_{6,T}}
 \int_0^t (t-s)^{-3/4} f(s){\rm\,d}s.
\]
Here $c$ is the $L^{\frac65}(\R^3)$-norm of the kernel of $e^{\Delta}{\rm div\,}$.
But $f\in L^{\frac43,2}(\R)$ by H\"older inequality in Lorentz spaces,
with norm controlled by the norm of $\theta$ in $L^2(0,T;L^\frac32(\R^3))$, independently of~$T$. 
Moreover, $|\!\cdot\!|^{-3/4}\in L^{\frac43,\infty}(\R)$, hence
$t\mapsto \|C(u,\theta)(t)\|_{L^{3/2}}{\mathds{1}}_{\R^+}(t)$ belongs to $L^{2,2}(\R)=L^2(\R)$
by Young-O'Neil inequality (see~\cite[Theorem~2.3]{Lem02}). 
\end{proof}

Let us now recall the local existence theorem in~\cite[Theorem~2.4]{BraH}.

\begin{theorem}[See \cite{BraH}]
\label{th:brahe}
If $3<p<\infty$, $\frac32<q<3$ and $\frac23<\frac1{p}+\frac1{q}$, and if
$(u_0,\theta_0)$ belongs to the closure of the Schwartz class $\mathscr{S}(\R^3)^3\times\mathscr{S}(\R^3)$ 
in the space $B^{-(1-3/p)}_{p,\infty}(\R^3)^3\times B^{-3(1-1/q)}_{q,\infty}(\R^3)$,
with ${\rm div}\,u_0=0$, then there exists $T>0$ and a solution $(u,\theta)$ to~\eqref{B} such that
\[
(u,\theta)\in \bigl(Z_{p,T}\cap Z_{\infty,T}\bigr)\times
  \bigl(Y_{q,T}\cap Y_{\infty,T}\bigr).
  \]
Moreover,
\[
\|u\|_{Z_{p,T}\cap Z_{\infty,T}}+\|\theta\|_{Y_{q,T}\cap Y_{\infty,T}}\xrightarrow[T\to0]{}0.
\]
Furthermore, if
$u_0\in L^3(\R^3)^3\subset B^{-(1-3/p)}_{p,\infty}$,
then $u\in \C([0,T],L^3(\R^3)^3)$ and if $\theta_0\in L^1(\R^3)\subset B^{-3(1-1/q)}_{q,\infty}$ then $\theta_0\in \C([0,T],L^1(\R^3))$.
\end{theorem}

Let us observe that the perturbation method used in~\cite{BraH} to establish Theorem~\ref{th:brahe} provides the 
well-posedness only in the space where the solution is constructed.

\begin{proof}[Proof of Theorem~\ref{th:existence}]
Under the assumptions of the first item of Theorem~\ref{th:existence}, we have
$u_0\in L^3(\R^3)^3$ and  $\theta_0\in B^{-1}_{\frac32,2}(\R^3)$,
which is continuously embedded in $B^{-3(1-1/q)}_{q,\infty}$, for all $q>3/2$.
Moreover, the Schwartz class is dense both in $L^3$ and 
in Besov spaces with finite third index.
Therefore we may apply Theorem~\ref{th:brahe}. Choosing, for example,
$p=6$ and $q=2$ we obtain the existence, for some $T>0$, of a solution 
$(u,\theta)\in Z_{6,T}\times Y_{2,T}$, such that
$u\in \C([0,T],L^3(\R^3)^3)$ and 
\begin{equation}
\label{eq:vani}
\|u\|_{Z_{6,T}}+\|\theta\|_{Y_{2,T}}\xrightarrow[T\to0]{}0.
\end{equation}
By the Boussinesq equation~\eqref{def:sol}, we have $\theta=b+C(u,\theta)$.
Moreover, by the heat kernel characterisation of Besov spaces,
we deduce from the condition $\theta_0\in B^{-1}_{\frac32,2}(\R^3)$,
that $b=e^{t\Delta}\theta_0\in L^2(0,T;L^\frac32(\R^3)^3)$.
Now, reducing if necessary the length of time interval where the solution is considered,  we can assume that $T$ is such that
$\kappa\|u\|_{Z_{6,T}}<1$. Hence, by Proposition~\ref{prop:lo}, we see that
the linear operator
$C(\cdot,u)\colon L^2(0,T;L^\frac32(\R^3)^3)\to L^2(0,T;L^\frac32(\R^3)^3)$
is bounded with norm less than~$1$.
Therefore, the operator $T:=I-C(\cdot,u)$ is invertible in such space.
But $T(\theta)=b$, hence $\theta=T^{-1}(b)\in  L^2(0,T;L^\frac32(\R^3)^3)$.

Let us prove the second assertion of Theorem~\ref{th:existence}. 
If $\theta_0$ belongs to the homogeneous Besov space $\dot B^{-1}_{3/2,2}(\R^3)$,
then $b\in L^2(0,\infty;L^\frac32(\R^3))$. Moreover, the norm
of~$b$ is controlled by  $\|\theta_0\|_{\dot B^{-1}_{3/2,2}(\R^3)}$ (and conversely).
If $\|u_0\|_{L^3}+\|\theta_0\|_{\dot B^{-1}_{3/2,2}(\R^3)}$ is smaller than
a suitable absolute constant 
(or, more in general, if $\|u_0\|_{\dot B^{-(1-3/p)}_{p,\infty}(\R^3)}
+\|\theta_0\|_{\dot B^{-3(1-1/q)}_{q,\infty}(\R^3)}$ is smaller than a constant
depending only on $p$ and $q$, where $p$ and $q$ are
as in Theorem~\ref{th:brahe}),
then the estimates in~\cite{BraH} provide the global existence of the solution,
with $u\in \C_b(0,\infty;L^3(\R^3))$. Moreover,
$\|u\|_{Z_{6,\infty}}$ is controlled by the size of the initial data $(u_0,\theta_0)$
in $L^3(\R^3)^3\times \dot B^{-1}_{3/2,2}(\R^3)$.
Therefore, $\kappa\|u\|_{Z_{6,\infty}}$ can be assumed to be smaller than~$1$.
Then the above argument applies with $T=+\infty$. This completely establishes Theorem~\ref{th:existence}.
\end{proof}


\end{document}